\documentclass[a4paper,12pt,reqno]{amsart}

\usepackage[utf8]{inputenc}
\usepackage{verbatim}
\usepackage{amsthm}
\usepackage{amsmath}
\usepackage{amssymb}
\usepackage{enumerate}
\usepackage[active]{srcltx}
\numberwithin{equation}{section}
\usepackage{tikz}
\usetikzlibrary{calc}

\usepackage{t1enc}

\newcommand{\mc}[1]{\mathcal{#1}}
\newcommand{\mbb}[1]{\mathbb{#1}}

\newtheorem{theorem}{Theorem}[section]
\newtheorem{lemma}[theorem]{Lemma}
\newtheorem{corollary}[theorem]{Corollary}
\newtheorem{claim}{Claim}[theorem]

\newtheorem{ntheorem}{Theorem}

\numberwithin{theorem}{section}

\theoremstyle{definition}
\newtheorem{definition}[theorem]{Definition}

\theoremstyle{remark}
\newtheorem{remark}{Remark}

\newcommand{\setm}{\setminus}
\newcommand{\empt}{\emptyset}
\newcommand{\subs}{\subset}
\newcommand{\dom}{\operatorname{dom}}
\newcommand{\ran}{\operatorname{ran}}
\newcommand{\cf}{\operatorname{cf}} 
\newcommand{\supp}{\operatorname{supp}}

\def\br#1;#2;{\bigl[ {#1} \bigr]^ {#2} }

\newcommand{\term}[1]{TERM({#1})}

\newcommand{\fx}[2]{#1\cup\{#2\}}
\newcommand{\extend}[1]{{#1}^+}
\newcommand{\trans}[2]{\rho_{#1,#2}}
\newcommand{\model}[1]{\mc M_{#1}}
\newcommand{\trace}[2]{#1\lceil #2}
\newcommand{\itrace}[2]{#1\lceil^* #2}
\newcommand{\perm}[1]{\operatorname{S}(#1)}
\newcommand{\parcperm}[2]{\operatorname{S}_{#2}(#1)}

\newcommand{\bijp}{\operatorname{{Bij_p}}}
\newcommand{\fin}{\operatorname{Fin}}

\def\<{\left\langle}
\def\>{\right\rangle}

\newcommand{\gen}[1]{\<#1\>_{gen}}

\author[S. Shelah]{Saharon Shelah}
\address
      { Institute of Mathematics\\Hebrew University, Jerusalem  }
\email{}

\author[L. Soukup]{Lajos Soukup}
\address
      { Alfr{\'e}d R{\'e}nyi Institute of Mathematics, Budapest, Hungary  }
\email{soukup@renyi.hu}
\urladdr{http://www.renyi.hu/$\tilde{}$soukup}

\subjclass[2000]{03E35, }
\keywords{permutation group, transitive, homogeneous}
\title{On $\kappa$-homogeneous, but not $\kappa$-transitive permutation groups}
	 \thanks{The first author was supported by European Research Council, grant no. 338821. \\ 
	 Publication Number F1886}
\thanks{The second author was supported  by NKFIH grants no. K113047 and K129211}
\date{Nov 24, 2019.}

\begin{document}

\begin{abstract}

	A permutation group  $G$ on a set $A$ is  {\em ${\kappa}$-homogeneous}
	iff for all  $X,Y\in\br A;{\kappa};$ with
	$|A\setm X|=|A\setm Y|=|A|$ there is
	a $g\in G$ with $g[X]=Y$.
	$G$ is 
  {\em ${\kappa}$-transitive}
	iff for any injective function  $f$  with $\dom(f)\cup\ran(f)\in \br A;\le {\kappa};$  
	and $|A\setm \dom(f)|=|A\setm \ran(f)|=|A|$ there is
	a $g\in G$ with $f\subs g$.

	Giving a partial answer to a question of P. M. Neumann \cite{N} we show that
	there is an ${\omega}$-homogeneous but not ${\omega}$-transitive 
	permutation group on a cardinal ${\lambda}$ provided 
	\begin{enumerate}[(i)]
	\item ${\lambda}<{\omega}_{\omega}$, or 
	\item $2^{\omega}<{\lambda}$, and
	${\mu}^{\omega}={\mu}^+$ and  $\Box_{\mu}$ hold  
	for each ${\mu}\le{\lambda}$ with ${\omega}=\cf({\mu})<{{\mu}}$, or
	\item our model was obtained by adding  ${\omega}_1$ many Cohen  generic reals 
	to some ground model.  
  \end{enumerate} 
	
For ${\kappa}>{\omega}$ we give a method to  construct large   ${\kappa}$-homogeneous, but not 
${\kappa}$-transitive permutation groups. Using this method  we show that 
there exists  ${\kappa}^+$-homogeneous, but not ${\kappa}^+$-transitive
permutation groups on ${\kappa}^{+n}$ for each infinite cardinal ${\kappa}$
and natural number $n\ge 1$ provided $V=L$.
\end{abstract}

 \maketitle

\section{Introduction}
 
Denote by $\perm A$ the group of all permutations of
the set $A$.
The subgroups of $\perm A$ are  called
{\em permutation groups on $A$}. 

We say that
a permutation group  $G$ on $A$ is  {\em ${\kappa}$-homogeneous}
iff for all  $X,Y\in\br A;{\kappa};$ with
$|A\setm X|=|A\setm Y|=|A|$ there is
a $g\in G$ with $g[X]=Y$.

We say that
a permutation group  $G$ on $A$ is  {\em ${\kappa}$-transitive}
iff for any injective function  $f$  with $\dom(f)\cup\ran(f)\in \br A;\le {\kappa};$  
and $|A\setm \dom(f)|=|A\setm \ran(f)|$ there is
a $g\in G$ with $f\subs g$.

In this paper we give a partial answer to the following question which  
was raised by 
P.N. Neumann   in \cite[Question 3]{N}: 
\begin{enumerate}[(i)]
\item[] {\em Suppose that ${\kappa}\le {\lambda}$ are infinite  cardinals. 
Does there exist a permutation group
on ${\lambda}$ that are ${\kappa}$-homogeneous, but not ${\kappa}$-transitive?}
\end{enumerate}

In section 2 we show that there exist ${\omega}$-homogeneous, but not ${\omega}$-transitive permutation groups
on ${\lambda}<{\omega}_{\omega}$ in ZFC, and on any infinite ${\lambda}$ if $V=L$ (see Theorem \ref{tm:countable-main}).

In section 3 we develop a general  method to obtain large ${\kappa}$-homogeneous, but not ${\kappa}$-transitive
permutation groups for arbitrary ${\kappa}\ge {\omega}$ (see Theorem \ref{tm:-main-uncountable-transitive-group}).
Applying our method we show that if ${\kappa}^{\omega}={\kappa}$, ${\lambda}={\kappa}^{+n}$ for some $n<{\omega}$,
and $\Box_{{\nu}}$ holds for each  ${\kappa}\le {\nu}<{\lambda}$,
then there is a ${\kappa}$-homogeneous, but not ${\kappa}$-transitive 
permutation group on ${\lambda}$ (Corollary \ref{cor:largekappa}).

Finally in section 4, using some lemmas from section 3,  we prove that after adding ${\omega}_1$ Cohen reals in the generic extension 
for each infinite ${\lambda}$ there exist ${\omega}$-homogeneous, but not ${\omega}$-transitive permutation groups
on ${\lambda}$  (Theorem \ref{tm:ohnot_con}). 

\medskip

Our notation is standard. 

\begin{definition}
	If ${\lambda}$ is fixed and $f\in S(A)$ for some $A\subs {\lambda}$,
	we take $$\extend{f}=f\cup (\operatorname{id}\restriction{({\lambda}\setminus A)})\in S({\lambda}).$$	
	\end{definition}

Given a family of functions, $\mc G$, we say that a function $y$ is
{\em $\mc G$-large}  	 iff	
\begin{displaymath}
|y\setm \bigcup \mc H|=|y|
\end{displaymath}
for each finite $\mc H\subs \mc G$.

We say that a permutation group on $A$ is {\em ${\kappa}$-intransitive}
iff there is 
a  $G$-large injective function 
$y$ with $\dom(y)\cup\ran (y)\in \br A;{\kappa};$ and 
$|A \setm \dom(y)|=|A\setminus \ran (y))|=|A|$.

A ${\kappa}$-intransitive group is clearly not ${\kappa}$-transitive.

\section{${\omega}$-homogeneous but not ${\omega}$-transitive}

\begin{definition}
Given a set $A$
we say that a family $\mc A\subs \br {A};{\omega};$ is
{\em nice on ${A}$}  iff $\mc A$ has an enumeration
 $\{A_{\alpha}:{\alpha}<{\mu}\}$ such that 
\begin{enumerate}[({N}1)]
	\item $\mc A$ is cofinal in  $\<\br {A};{\omega};,\subs\>$, 
	\item \label{n2}for each ${\beta}<{\mu}$ there is a countable set $I_{\beta}\in \br {\beta};{\omega};$
	such that 
	 for all ${\alpha}<{\beta}$ there is 
	 a finite set $J_{{\alpha},{\beta}}\in \br I_{\beta};<{\omega};$ such that 
	\begin{displaymath}
	A_{\alpha}\cap A_{\beta}\subset \bigcup_{{\zeta}\in J_{{\alpha},{\beta}}}A_{\zeta}.	
	\end{displaymath}
\end{enumerate}
\end{definition}

\begin{theorem}\label{tm:nice2ord}
Assume that  ${\lambda}$ is an infinite cardinal, and   
$\mc A\subs \br {\lambda};{\omega};$ is   a nice family on ${\lambda}$.
Then for each $A\in \mc A$ there is an ordering $\le_A$ on $A$ such that 
\begin{enumerate}[(1)]
	\item $tp(A,\le_A)={\omega}$ for each $A\in \mc A$,
	\item \label{fin-part} if $A,B\in \mc A$, then there is a partition $\{C_i:i<n\}$
	of $A\cap B$ into finitely many subsets such that 
	$\le_A\restriction C_i=\le_B\restriction C_i$ for all $i<n$.
\end{enumerate}

\end{theorem}

\begin{proof}
Fix an enumeration $\{A_{\beta}:{\beta}<{\mu}\}$ of $\mc A$
witnessing that $\mc A$ is nice.

We will define $\le_{A_{\beta}}$ by induction on ${\beta}<{\mu}$.

Assume that $\le_{A_{\alpha}}$ is defined for ${\alpha}<{\beta}$.

By (N\ref{n2}) we can fix a countable set   $I_{\beta}=\{{\beta}_i:i<{\omega}\}\in \br {\beta};{\omega};$
	such that 
	 for all ${\alpha}<{\beta}$ there is 
	 $n_{\alpha}<{\omega}$ such that 
	\begin{displaymath}
	A_{\alpha}\cap A_{\beta}\subset \bigcup_{i<n_{\alpha}}A_{{\beta}_i}.	
	\end{displaymath}
Choose an order  $\le_{A_{\beta}}$  on $A_{\beta} $ such that 
\begin{enumerate}[(i)]
	\item  for each $i<{\omega}$ writing 
	$D_i=A_{{\beta}_i}	\setminus\bigcup_{j<i}A_{{\beta}_j}$ we have 
	\begin{displaymath}
	\le_{A_{\beta}}\restriction (A_{\beta}\cap D_i)\quad=\quad
	\le_{A_{{\beta}_i}}\restriction (A_{\beta}\cap D_i);
	\end{displaymath}
\item $tp(A_{\beta},\le_{A_{{\beta}}})={\omega}$.
\end{enumerate}

\medskip

By induction on ${\beta}$ we show that (\ref{fin-part}) holds for ${\beta}$.

Assume that (\ref{fin-part}) holds for ${\beta}'<{\beta}$.

To check (\ref{fin-part}) for ${\beta}$ fix ${\alpha}<{\beta}$.

To define $\le _{\beta}$ we considered a set
$I_{\beta}=\{{\beta}_i:i<{\omega}\}\in \br {\beta};{\omega};$
	such that 
	  we had 
	 $n_{\alpha}<{\omega}$ with 
	\begin{displaymath}
	A_{\alpha}\cap A_{\beta}\subset \bigcup_{i<n_{\alpha}}A_{{\beta}_i}.	
	\end{displaymath}
For $i<n_{\alpha}$ let $C'_i=A_{\alpha}\cap A_{\beta}\cap D_i$, where $D_i=A_{{\beta}_i}	\setminus\bigcup_{j<i}A_{{\beta}_j}$ .
Then $\{C'_i:i<n_{\alpha}\}$ is a partition of $A_{\alpha}\cap A_{\beta}$
and  $$\le _{A_{\beta}}\restriction C'_i=\le_{A_{{\beta}_i}}\restriction C'_i$$
by (i).
By the inductive hypothesis, 
$A_{{\beta}_i}\cap A_{\alpha}$ has a partition into finitely many pieces
$\{C_{i,j}:j<k_i\}$ such that 
$\le_{A_{\alpha}}\restriction C_{i,j}=\le_{A_{\beta_i}}\restriction C_{i,j}$
Then the partition
\begin{displaymath}
\{C'_i\cap C_{i,j}:i<n, j<k_i\}	
\end{displaymath}
of $A_{\alpha}\cap A_{\beta}$ works for ${\alpha}$  and ${\beta}$.
Indeed,
\begin{displaymath}
\le_{ A_{\alpha}}\restriction	 C'_i\cap C_{i,j}\quad =
\quad \le_{ A_{{\beta}_i}}\restriction	 C'_i\cap C_{i,j}\quad =
\quad \le_{A_{{\beta}}}\restriction	 C'_i\cap C_{i,j}.
\end{displaymath}
\end{proof}

\begin{theorem}\label{tm:ord2group}
Assume that  ${\lambda}$ is an infinite cardinal, 
 $\mc A\subs \br {\lambda};{\omega};$  is a cofinal family
and  
 for each $A\in \mc A$ we have an ordering $\le_A$ on $A$ such that 
\begin{enumerate}[(1)]
	\item $tp(A,\le_A)={\omega}$ for each $A\in \mc A$,
	\item if $A,B\in \mc A$, then there is a partition $\{C_i:i<n\}$
	of $A\cap B$ into finitely many subsets such that 
	$\le_A\restriction C_i=\le_B\restriction C_i$ for all $i<n$.
\end{enumerate}
Then there is a permutation group on ${\lambda}$ that is ${\omega}$-homogeneous
and  ${\omega}$-intransitive.	
\end{theorem}

\begin{proof}
For $A\in \mc A$ let
\begin{multline}\notag
\mathcal G_A=\{\extend f\in \perm {\lambda}: 
f\in \perm A \land 
\text{there is a finite partition $\{C_i:i<n\}$ of $A$}\\\text{ such that $f\restriction C_i$ is
$\le_A$-order preserving}\}.
\end{multline}
Let $G$ be the permutation group on ${\lambda}$ generated by 
\begin{displaymath}
\bigcup\{\mc G_A:A\in \mc A\}.
\end{displaymath}

\begin{claim}
$G$ is ${\omega}$-homogeneous.
\end{claim}

Indeed, let  $X,Y\in\br{\lambda};{\omega};$ with
$|{\lambda}\setm X|=|{\lambda}\setm Y|={\lambda}$. 
Pick $A\in \mc A$ such that $X\cup Y\subs A$ and $|A\setm X|=|A\setm Y|={\omega}$.

Let $c$ be the unique $\le_A$-monotone bijection between $X$ and $Y$
and $d$ be the unique $\le_A$-monotone bijection between 
$A\setm X$ and $A\setm Y$. Then taking $g=c\cup d$ we have 
$\extend g\in \mc G_A\subs G$
and $\extend g[X]=Y$.

\begin{claim}
$G$ is ${\omega}$-intransitive.
\end{claim}

Pick $A\in \mc A$ and choose $B\in \br A;{\omega};$
such that $|A\setm B|={\omega}$.

Let $b_0,b_1,\dots$ be the $\le_A$-increasing enumeration of $B$.
Define a bijection $y:B\to {\omega}$ as follows:
for $i<{\omega}$ and $j<2^i$ let
$$y(b_{2^i+j})=b_{2^{i+1}-j}.$$ 
Observe that if $c$ is $\le_A$-monotone then
\begin{displaymath}
|\{i<{\omega}:|\{j<2^i:c(b_{2^i+j})=r(b_{2^i+j})\}|\ge 2\}|\le 1.
\end{displaymath}
Indeed, if $|\{j<2^i:c(b_{2^i+j})=y(b_{2^i+j})\}|\ge 2$,
then $c$ should be $\le_A$-decreasing, and if
 $|\{i:\{j<2^i:c(b_{2^i+j})=y(b_{2^i+j})\}\ne \empt\}|\ge 2$,
 then $y$ should be $\le_A$-increasing.
\smallskip

So $y$ can not be covered by finitely many $\le_A$-monotone functions.
But for any   $h\in G$, $h\cap (A\times A)$  can be covered by finitely many 
$\le_A$-monotone functions  by (2) and by the construction of $G$. 

Thus  $y$ is $G$-large.
\end{proof}

To obtain nice families we  recall some topological results. 
We say that a topological space  $X$ is {\em splendid} (see \cite{JNW}) iff it is
countably compact, locally compact, locally countable  such that 
$|\overline A|={\omega}$ for each $A\in \br X;{\omega};$. 

We need the following theorem:

\begin{ntheorem}[Juhasz, Nagy, Weiss, \cite{JNW}]
If 
	\begin{enumerate}[(i)]
		\item ${\kappa}<{\omega}_{\omega}$, or 
		\item $2^{\omega}<{\kappa}$, $\cf({\kappa})>{\omega}$ and
		${\mu}^{\omega}={\mu}^+$ and  $\Box_{\mu}$ hold  
		for each ${\mu}<{\kappa}$ with ${\omega}=\cf({\mu})<{{\mu}}$,
		\end{enumerate} 
then there is a splendid space $X$ of size ${\kappa}$.	
\end{ntheorem}

\begin{remark}
In 	\cite[Theorem 11]{JNW} the authors formulated 
a bit weaker result: {\it if $V=L$ and $\cf({\kappa})>{\omega}$
then there is a splendid space $X$ of size ${\kappa}$}.
However, 
to obtain that results they  combined  ``Lemmas 7, 9 and 16 
with the remark after Theorem 8''
and their arguments  used only the assumptions of 
 the theorem above. 
\end{remark}

\begin{lemma}\label{lm:splendid2nice}
If $X$ is a splendid space, $\mc U$ is the family of compact open 
subsets of $X$, and $Y\subs X$, then $\trace {\mc U}Y=\{U\cap Y: U\in \mc U\}$ is nice on $Y$.  		
\end{lemma}

\begin{proof}
Let $A\in \br Y;{\omega};$. Then $\overline A$ is countable, so it is 
compact. Since a splendid space is zero-dimensional, $A$ can be covered by 
finitely many compact open set, and so $A$ can be covered by 
an element of $\mc U$. Thus $\trace{\mc U}Y$ is  cofinal in 
$\<\br Y;{\omega};,\subset\>$.  

To check (N\ref{n2}) observe that every   $U\in \mc U$ is a countable compact space, 
so it is homeomorphic to a countable successor ordinal.  Thus $U$
has only countably many compact open subsets.
Hence $\mc U\restriction U$ is countable which implies  (N\ref{n2}) in the 
following stronger form:
\begin{enumerate}[({N}1${}^+$)]\addtocounter{enumi}{1}
	\item \label{n2x}for each ${\beta}<{\mu}$ there is a set $I_{\beta}\in \br {\beta};{\omega};$
	such that 
	 for all ${\alpha}<{\beta}$ there is 
	 ${\zeta}_{\alpha}\in I_{\beta}$ such that 
	\begin{displaymath}
	A_{\alpha}\cap A_{\beta}
	=A_{{\zeta}_\alpha}\cap A_{\beta}.
	\end{displaymath}
\end{enumerate}
\end{proof}

\begin{remark}
By \cite[Corollary 2.2]{JShS}, {\it if   
$({\omega}_{\omega+1}, {\omega}_{\omega})\to({\omega}_1,{\omega})$ holds,
then the cardinality of a splendid space is less than ${\omega}_{\omega}$}.
So we need some new ideas if we want to construct arbitrarily large nice
families in ZFC.	
\end{remark}

\begin{theorem}\label{tm:countable-main}
	If  ${\lambda}$ is an infinite cardinal, and 
	\begin{enumerate}[(i)]
		\item ${\lambda}<{\omega}_{\omega}$, or 
		\item $2^{\omega}<{\lambda}$,  and
		${\mu}^{\omega}={\mu}^+$ and  $\Box_{\mu}$ hold  
		for each ${\mu}\le {\lambda}$ with ${\omega}=\cf({\mu})<{{\mu}}$. 
		\end{enumerate} 
then there is an ${\omega}$-homogeneous and
${\omega}$-intransitive permutation group on ${\lambda}$.	
\end{theorem}

\begin{proof}
Applying 	the Juhasz-Nagy-Weiss theorem for ${\kappa}={\lambda}$ if 
$\cf({\lambda})>{\omega}$,
and for ${\kappa}={\lambda}^+$ if ${{\lambda}>\cf(\lambda})={\omega}$,
we obtain a splendid space on ${\kappa}\ge {\lambda}$. So, by 
Lemma \ref{lm:splendid2nice}, we obtain a nice family on $\mc A$ on ${\lambda}$.

Thus, putting together Theorems \ref{tm:nice2ord} and \ref{tm:ord2group}
we obtained the desired permutation group on ${\lambda}$.
\end{proof}

\section{${\kappa}$-homogeneous but not ${\kappa}$-transitive
for ${\kappa}>{\omega}$}

Write  $\trace{\mc A}X=\{A\cap X:A\in \mc A\}$ and
$\itrace{\mc A}X=\{\bigcap \mc A'\cap X:\mc A'\in \br \mc A;<{\omega};\}$.

\begin{definition}
Let ${\kappa}<{\lambda}$ be cardinals.
We say that a cofinal family $\mc A\subs \br {\lambda};{\kappa};$ 
is {\em locally small} iff $|\trace {\mc A}A|\le {\kappa}$
for all $A\in \mc A$.
\end{definition}

\begin{definition}
If $X,Y$ are subsets of ordinals with the same order types, then let 
$\trans XY $ be the unique order preserving bijection between $X$ and $Y$.	
\end{definition}

\begin{definition}\label{df:term-eval}
	If $\mc F$ is a set of functions, 	an {\em $\fx{\mc F}{x}$-term   $t$}  
	is a  sequence $\<h_0,\dots, h_{n-1}\>$,
	where  $h_i=x$ or $h_i=x^{-1}$  or $h_i=f_i$ or $h_i={f_i}^{-1}$ for some $f_i\in \mc F$. 
	If $g$ is function we use 
	$t[g]$ to denote the function $h'_0\circ h'_1\circ \dots \circ h'_{n-1}$,
	where 
	\begin{displaymath}
	h'_i=\left\{\begin{array}{ll}
	f_i&\text{if $h_i=f_i$, }\\	
	f^{-1}_i&\text{if $h_i=f^{-1}_i$,}\\	
	g&\text{if $h_i=x$,}\\	
	g^{-1}&\text{if $h_i=x^{-1}$.}\\	
	\end{array}  \right .
	\end{displaymath}

If $\mc H$ is a set of $\fx{\mc F}{x}$-terms, then write
\begin{displaymath}
\mc H[g]=\{t[g]:t\in H\}.
\end{displaymath}

	We say that an $\fx{\mc F}{x}$-term $t$ is an {\em $\mc F$-term} iff 
	neither  $x$ nor $x^{-1}$
	are in the $t$. If $t$ is a $\mc F$-term, then the function $t[g]$ does not depends on $g$,
	so we will write $t[\ ]$ instead of $t[g]$ in that situation.
	
	We say that a term $t'$ is a {\em subterm} of a term $t=\<h_0,\dots, h_{n-1}\>$ iff   
	$t'=\<h_{i_0}, h_{i_1},\dots , h_{i_k}\>$,
	where $i_0<i_1<\dots <i_{k}<n$.
	
	The set of all $\fx{\mc F}{x}$-terms is denoted by $\term {\fx{\mc F}{x}}$.

	The set of all $\mc F$-terms is denoted by $\term {\mc F}$.
	
	\end{definition}

\begin{theorem}\label{tm:-main-uncountable-transitive-group}
Assume that $2^{\kappa}={\kappa}^+$ and 
there is a cofinal, locally small family $\mc A\subs  \br 	{\lambda};{\kappa};$.
Then there is a permutation group $G$ on ${\lambda}$ which is ${\kappa}$-homogeneous,
but not ${\kappa}$-transitive.	
\end{theorem}

Before proving this theorem we need some preparation.

\begin{lemma}\label{lm:induction-step-step}
	Assume that 
\begin{enumerate}[(1)]
\item  ${\lambda}$ is a cardinal, $\mc H$ is a finite set of  $\fx{S({\lambda})}{x} $-terms, and 
$\mc H$ is closed for  subterms,
\item $g$ is an injective function, 
$\dom(g)\cup\ran(g)\subs {\lambda}$, 
\item \label{ayaundef} ${\alpha},{\alpha}^*\in {\lambda}$ such that 
\begin{displaymath}
\<{\alpha},{\alpha}^*\>\notin \bigcup \mc H[g],  
\end{displaymath}
\item ${\zeta}_0\in {\lambda}\setminus \dom(g)$ and ${\zeta}_1\in 
{\lambda}\setminus \ran(g)$,
\item \label{etagengeric} ${\eta}_0\in {\lambda}\setminus \ran(g)$ 
and ${\eta}_1\in {\lambda}\setminus \dom(g)$
such that 
\begin{displaymath}
{\eta}_0,{\eta}_1\notin\{t[g]({\alpha}), t[g]^{-1}({\alpha}^*) :t\in \mc H\}.
\end{displaymath}
\end{enumerate}	
Let $g_0=g\cup\{\<{\zeta}_0,{\eta}_0\>\}$ and 
$g_1=g\cup\{\<{\eta}_1,{\zeta}_1\>\}$.
Then  
\begin{displaymath}
\<{\alpha},{\alpha}^*\>\notin \mc H[g_0]\cup \mc H[g_1].  
\end{displaymath}
\end{lemma}

\begin{proof}
We prove only $\<{\alpha},{\alpha}^* \>\notin \mc H[g_0]$.  
The proof of the other statement is similar. 

Assume on the contrary that $\<{\alpha},{\alpha}^* \>\in \mc H[g_0]$.

Pick  the shortest term $t=\<f_0,\dots, f_n\>$  from $\mc H$ such that
$t[g_0]({\alpha})={\alpha}^*$.

Write ${\alpha}_{n+1}={\alpha}$ and 
${\alpha}_{i}=\<f_{i},\dots ,f_n\>[g_0]({\alpha})$
for $0\le i\le n$.  Hence ${\alpha}_{0}={\alpha}^* $.

Let $i$ maximal such that ${\alpha}_i$ is $	{\zeta}_0$ or ${\eta}_0$.
Since $t[g]({\alpha})$ can not be ${\alpha}^* $ by (\ref{ayaundef}), $i$ is defined.

Since ${\alpha}_i=\<f_i,\dots ,f_{n}\>[g]({\alpha})$, it follows that 
${\alpha}_i\ne {\eta}_0$ by (\ref{etagengeric}). So ${\alpha}_i={\zeta}_0$.

Let $j$ minimal such that ${\alpha}_j$ is $	{\zeta}_0$ or ${\eta}_0$.
Since ${\alpha}_j=(\<f_0,\dots ,f_{j-1}\>[g])^{-1}({\alpha}^* )$, 
it follows that 
${\alpha}_j\ne {\eta}_0$  by (\ref{etagengeric}). So ${\alpha}_j={\zeta}_0$  by (\ref{etagengeric}).
Thus ${\alpha}_i={\alpha}_j={\zeta}_0$,  and so 
\begin{displaymath}
{\alpha}^* =\<f_0,\dots, f_{j-1},f_i,\dots, f_n\>[g_0]({\alpha}).
\end{displaymath}
Since $j< i$, the term $t'=\<f_0,\dots, f_{j-1},f_i,\dots, f_n\>$ is shorter 
than $t$ and still ${\alpha}^* =t'[g_0]({\alpha})$.
So the length of $t$ was not minimal. Contradiction.
\end{proof}

\begin{lemma}\label{lm:induction-step}
	Assume that 
	\begin{enumerate}[(1)]
	\item $y\in \perm {\kappa}$,
	\item $A\in \br {\lambda};{\kappa};$,
	and $B,C\in \br A;{\kappa};$ such that 
	$|A\setm B|=|A\setm C|={\kappa}$,
	\item $\mc F\in \br {\perm {\lambda}};{\kappa};$ such that  
	\begin{displaymath}
	 	|y\setminus \bigcup \mc H[\ ]|={\kappa}
	 	\end{displaymath}
    whenever $\mc H$ is a finite set of $\mc F$-terms. 
	\end{enumerate}
	Then there is $g\in \perm A$ such that 
	\begin{enumerate}[(i)]
	\item 	$g[B]=C$,
	\item 		
	\begin{displaymath}
		|y\setminus \mc H[\extend{g}]|={\kappa}
		\end{displaymath}	
whenever 	$\mc H$ is a finite set of $\fx{\mc F}{x}$-terms.
	\end{enumerate}
	\end{lemma}

	\begin{proof}[Proof of Lemma \ref{lm:induction-step}]
		Write
		\begin{displaymath}
			\mbb {TASK}_0= A\times \{\dom,\ran\}\ \text{and}\  
			\mbb {TASK}_1=\br \term{\fx{\mc F}{x}};<{\omega};\times {\kappa}.
			\end{displaymath}

Let $\{I_0,I_1\}\in \br {\br {\kappa};{\kappa};};2;$ be a partition of ${\kappa}$,
and fix  enumerations $\{T_i:i\in I_0\}$ of $\mbb{TASK}_0$, 
and $\{T_i:i\in I_1\}$ of  $\mbb{TASK}_1$.

By transfinite induction, for  $i<{\kappa}$
we will construct a function $g_i$ 
and if $i=j+1$ for some $j\in K_1$ then we also pick an ordinal ${\alpha}_{j+1}\in {\kappa}$ for 
such that 
\begin{enumerate}[(a)]
\item $g_i$ is an injective function, $\dom(g_i)\cup \ran(g_i)\subs A$,
\item $g_i[B]\subs C$ and  $g_i[A\setminus B]\subs A\setminus C$;
\item $|g_i|\le i$;
\item if $i=j+1$, $j\in I_0$ and $T_j=\<{\zeta},\dom\>$, then 
${\zeta}\in \dom(g_i)$;
\item if $i=j+1$, $j\in I_0$ and $T_j=\<{\zeta},\ran\>$, then 
${\zeta}\in \ran(g_i)$;

\item \label{y-Hgiplus} if $i=j+1$, $j\in I_1$  and $T_j=\<\mc H_j,\chi_j\>$,
then 
\begin{enumerate}[(i)]
\item ${\alpha}_{j+1}\in {\kappa}\setm\{{\alpha}_{j'+1}:j'\in I_1\cap j\}$, and
\item $t[{g_i}\cup \operatorname{id}_{{\lambda}\setm A}]({\alpha}_{j+1})$ is defined
and  $t[{g_i}\cup \operatorname{id}_{{\lambda}\setm A}]({\alpha}_{j+1})\ne y({\alpha}_{j+1})$
for each $t\in \mc H_{j}$.
 \end{enumerate}
 
\end{enumerate} 

\medskip

Let $g_0=\empt$.

\medskip

If $i$ is limit, then let $g_i=\bigcup_{j<i}g_j$.

\medskip

Assume that $i=j+1$.

\begin{claim}\label{cl:yHgj}
	\begin{equation}\tag{\dag}\label{dag}
		|y\setm \bigcup\mc H[g_j\cup \operatorname{id}_{{\lambda}\setminus A	}]|={\kappa}.
		\end{equation}
		for each finite set $\mc H$ of $\fx{\mc F}{x}$-terms.
			\end{claim}

\begin{proof}[Proof of the Claim]
Fix $\mc H$. We can assume that $\mc H$ is closed for subterms.
	By (3) we have 
	$|y\setminus \bigcup \mc H[\ ]|={\kappa}$, and 
	 \begin{displaymath}\tag{$\circ$}
		y\cap \bigcup \mc H[\ ]=y\cap \bigcup 
		\mc H[\operatorname{id}_{{\lambda}\setminus A}]	 
	 \end{displaymath}
	 because $\mc H$ is closed for subterms.
	Since $|g_j|<{\kappa}$, we have 
\begin{displaymath}\tag{$\bullet$}
	|t[\operatorname{g_j\cup id}_{{\lambda}\setminus A}]\setminus
	t[\operatorname{id}_{{\lambda}\setminus A}]|<{\kappa}.
\end{displaymath}
for each $t\in \mc H$.
Putting together  $|y\setminus \bigcup \mc H[\ ]|={\kappa}$, $(\circ)$ and $(\bullet)$ we obtain (\ref{dag}).
	\end{proof}

\noindent {\bf Case 1.} {\it $j\in I_0$ and so 
$T_j=\<{\zeta}_j,x_j\>\in A\times\{\dom,\ran\}$.}

Assume  first that $x_j=\dom$.  If ${\zeta}_j\in \dom (g_j)$, let $g_i=g_j$.
If ${\zeta}_j\notin \dom (g_j)$, then 
pick ${\eta}\in C$ if ${\zeta}_i\in B$, and 
pick ${\eta}\in A\setminus C$ if ${\zeta}_i\in A\setminus B$
such that 
and ${\eta}\notin \ran(g_j)$.

Let $g_i=g_j\cup\<{\zeta}_i,{\eta}\>$.
Then
$g_i$ satisfies 
(a)--(f).

The case $x_j=\ran$ is similar.

\medskip

\noindent {\bf Case 2.} {\it $j\in I_1$ and so 
$T_j=\<\mc H_j,\chi_j\>\in \br \term{\fx{\mc F}{x}};<{\omega};\times {\kappa}.$}

We can assume that $\mc H_j$ is closed for subterms.

By Claim \ref{cl:yHgj}, we have 
\begin{displaymath}
|y\setm \bigcup \mc H_j[g_j\cup id_{({\lambda}\setminus A)} ]|={\kappa}.
\end{displaymath}

So we can pick ${\alpha}_{j+1}\in {\kappa}\setm\{{\alpha}_{j'+1}:j'\in I_1\cap j\}$
such that 
 \begin{enumerate}[(i)]
	\item[$(*)$] for each $t\in \mc H_{j}$ either  $t[{g_j}\cup \operatorname{id}_{{\lambda}\setm A}]({\alpha}_{j+1})$ 
	is undefined
 or	$t[{g_j}\cup \operatorname{id}_{{\lambda}\setm A}]({\alpha}_{j+1})\ne y({\alpha}_{j+1})$
	.
	 \end{enumerate}
	
Now in finitely many steps, using Lemma \ref{lm:induction-step-step},
we can extend the function $g_j$ to a function $g_i$ such that 
\begin{enumerate}[(i)]
	\item[$(*)$] $t[{g_i}\cup \operatorname{id}_{{\lambda}\setm A}]({\alpha}_{j+1})$ 
	is defined and 
 $t[{g_i}\cup \operatorname{id}_{{\lambda}\setm A}]({\alpha}_{j+1})\ne y({\alpha}_{j+1})$
	for each $t\in \mc H_{j}$.
	 \end{enumerate}

Indeed, if $t[{g'}\cup \operatorname{id}_{{\lambda}\setm A}]({\alpha}_{j+1})$ 
is not defined, where $t=\<t_0,\dots, t_n\>$ then there is $i<n$
such that  
either 
\begin{enumerate}[(1)]
\item[] ${\zeta}_i=\<t_{i+1},\dots, t_n\>[{g'}\cup \operatorname{id}_{{\lambda}\setm A}]({\alpha}_{j+1})$
is defined, 
$t_{i}=x$ and ${\zeta}_i\in A\setm  \dom(g')$ 
\end{enumerate}
or  
\begin{enumerate}[(1)]
	\item[] ${\zeta}_i=\<t_{i+1},\dots, t_n\>[{g'}\cup \operatorname{id}_{{\lambda}\setm A}]({\alpha}_{j+1})$
	is defined, $t_{i}=x^{-1}$ and ${\zeta}_i\in A\setm \ran(g')$. 
	\end{enumerate}
In both cases, using Lemma \ref{lm:induction-step-step}, 
we can extend $g'$ to $g''$ such that
$\<t_i,\dots, t_n\>[{g''}\cup \operatorname{id}_{{\lambda}\setm A}]({\alpha}_{j+1})$
is defined and $\<{\alpha}_{j+1},y({\alpha}_{j+1})\>\notin \bigcup\mc H_j[g''\cup id_{{\lambda}\setm A}]$.

\medskip

After the inductive construction,
the function $g=\bigcup_{i<{\kappa}}g_i$ meets the requirements.
	\end{proof}

	\begin{lemma}\label{lm:coftopairs}
		Assume that $2^{\kappa}={\kappa}^+$ and there is a cofinal, locally small subfamily $\mc C\subs \br {\lambda};{\kappa};$.
		Then there is a family $\mc D\subs \br {\lambda};{\kappa};\times 
		\br 	{\lambda};	{\kappa}; $  such that 
		\begin{enumerate}[(1)]
		\item if $\<A,B\>\in \mc D$, then $B\cup {\kappa}\subs A$ and 
		$|A\setminus B|={\kappa}$.
		\end{enumerate}
		Moreover,
		writing $\mc A=\{A:\<A,B\>\in \mc D\}$ and $\mc B=\{B: \<A,B\>\in \mc D\}$
		\begin{enumerate}[(1)]\addtocounter{enumi}{1}
		\item  $\mc A$ is a cofinal, locally small subfamily of $\br {\lambda};{\kappa};$,
		\item $\mc B$ is cofinal in $\<\br {\lambda};{\kappa};,\subset\>$,
		\item $\{X\subs {\kappa}: |X|=|{\kappa}\setminus X|={\kappa}\}\subs  \mc B$.
		\end{enumerate}
	\end{lemma}

	\begin{proof}[Proof of Lemma \ref{lm:coftopairs}] 
	Fix a locally small, cofinal subfamily  $\mc C\subs \br {\lambda};{\kappa};$.
	 We can assume that 
	$|\{C\in \mc C:D\subs C\}|=|\mc C|$ for all 
	$D\in \br {\lambda};{\kappa};$.  
	
	Write ${\mu}=|\mc C|$. Then $2^{\kappa}={\kappa}^+\le {\mu}$.
	So we can construct $\mc D$ by induction such that 	$\mc A\subs \mc C$, ${\kappa}\subs \bigcap \mc A$  
	and $\mc B=\mc C\cup\{X\subs {\kappa}: |X|=|{\kappa}\setminus X|={\kappa}\}$. 
	\end{proof}

After that preparation we prove the main theorem of this section.

\begin{proof}[Proof of Theorem \ref{tm:-main-uncountable-transitive-group}]

Fix $\mc D$, $\mc A$ and $\mc B$ as in Lemma \ref{lm:coftopairs}.

For $\<A,B\>\in \mc D$ consider the structure  
$\model {\<A,B\>}=\<A,<,B,\{A\cap X:A\in \mc A\}\>$.

Fix $\mc D'\in \br \mc D;{\kappa}^+;$ such that 
writing $\mc A'=\{A':\<A',B'\>\in \mc D'\}$ and $\mc B'=\{B': \<A',B'\>\in \mc D'\}$
we have
\begin{enumerate}[(a)]
\item $\forall \<A,B\>\in \mc D$
$\exists \<A',B'\>\in \mc D'$ such that $\trans{A}{A'} $ is an isomorphism
between $\model {\<A,B\>}$ and  $\model {\<A',B'\>}$. 
\item $\{X\subs {\kappa}: |X|=|{\kappa}\setminus X|={\kappa}\}\subs  \mc B'$.
\end{enumerate}

Pick $K\in \br {\kappa};{\kappa};$ with $|{\kappa}\setm K|={\kappa} $.
Choose $y \in S({\kappa})$ such that 
$y({\alpha})\ne {\alpha}$
for each  ${\alpha}\in {\kappa}$.

\begin{lemma}[Key lemma]\label{lm:keylemma}
There are functions $\mc F=\{f_{\<A,B\>}:\<A,B\>\in \mc D'\}$
such that 
\begin{enumerate}[(a)]
\item $f_{\<A,B\>}\in \perm A$,
\item $f_{\<A,B\>}[B]=K$,
\end{enumerate}	
moreover, taking
\begin{align*}
	\mc S=\big\{\trans{C_0}{C_1}: 
\<A_0,B_0\>,\<A_1,B_1\>\in \mc D', C_0\in \itrace{\mc A}{A_0},
C_1\in \itrace{\mc A}{A_1},&\\ \trans{C_0}{C_1}[\trace{\mc A}{ C_0}]&= 
\trace{\mc A}{C_1}\},
\end{align*}
if $\mc H$ is a finite collection of  $\mc F\cup \mc S$-terms, then 
\begin{displaymath}
|y\setm \bigcup\mc H[\ ]|={\kappa}.
\end{displaymath}
\end{lemma}

Before proving the Key lemma, we show how the Key Lemma 
completes the proof of Theorem \ref{tm:-main-uncountable-transitive-group}.

So assume that the Key lemma holds. 

For each $\<A,B\>\in \mc D$ pick $\<A',B'\>\in \mc D'$ such that 
$\trans{A}{A'} $ is an isomorphism
between $\model {\<A,B\>}$ and  $\model {\<A',B'\>}$.
We assume that $\<A',B'\>=\<A,B\>$  for  $\<A,B\>\in \mc D'$.

Let 
\begin{displaymath}
g_{\<A,B\>}=\trans{A'}{A}\circ f_{\<A',B'\>}\circ \trans{A}{A'}
\in S(A).
\end{displaymath}
Let $G$ be the permutation group on ${\lambda}$ generated by 
\begin{displaymath}
\mc G=\{\extend{g_{\<A,B\>}}:\<A,B\>\in \mc D\}.
\end{displaymath}

\begin{lemma}\label{lm:homog}	
$G$ is ${\kappa}$-homogeneous.	
\end{lemma}	

\begin{proof}[Proof of Lemma \ref{lm:homog}]
It is enough to show that for each  $X\in \br {\lambda};{\kappa};$
there is $g\in G$ with $g[X]=K$.

So fix  $X\in \br {\lambda};{\kappa};$. Pick $\<A,B\>\in \mc D$ such that $X\subs B$.

Then  
\begin{align*}
	Z=g_{\<A,B\>}[X] \subs g_{\<A,B\>}[B]=&(\trans{A'}{A}\circ f_{\<A',B'\>}\circ \trans{A}{A'})[B]\\
	=&(\trans{A'}{A}\circ f_{\<A',B'\>})[B']=\trans{A'}{A}[K]=K.
\end{align*}
Since $|Z|=|{\kappa}\setm Z|={\kappa}$, there is $C$ such that  
$\<C,Z\>\in \mc D'$.
Then  $f_{\<C,Z\>}[Z]=K$. Thus $\extend{g_{\<C,Z\>}}[Z]=K$
because $\<C',Z'\>=\<C,Z\>$ and so $f_{\<C,Z\>}=g_{\<C,Z\>}$.

Thus  $K=(\extend{g_{\<C,Z\>}}\circ \extend{g_{\<A,B\>}})[X]$. 
\end{proof}

\begin{lemma}\label{lm:nottr}	
	$G$ is not ${\kappa}$-transitive.
	\end{lemma}	
	
	\begin{proof}[Proof of Lemma \ref{lm:nottr}]
	We prove that $y\not\subs h$ for any $h\in G$.

	Assume that 
	\begin{displaymath}
	h=(g_0^+)^{\ell_0}\circ(g_1^+)^{\ell_1}\circ
	\dots\circ (g_{n-1}^+)^{\ell_{n-1}}, 
	\end{displaymath}
where $g_i=g_{\<A_i,B_i\>}=\rho_{A'_i,A_i}\circ f_{A'_i, B'_i}\circ \rho _{A_i,A'_i}$	
and $\ell_i\in \{-1,1\}$ for $i<n$.

Since $g_i^+\setm g_i$ is the identity function on ${\lambda}\setm A_i$,
we have 
\begin{multline*}
h\subs \bigcup 
\{(g_{i_0})^{\ell_{i_0}}\circ(g_{i_1})^{\ell_{i_1}}\circ
\dots\circ (g_{i_{k-1}})^{\ell_{i_{k-1}}}:\\ k<n, i_0<i_1<\dots<i_{k-1}<n
	\}.	
\end{multline*}

Fix $k\le n$ and   $i_0<i_1<\dots<i_{k-1}<n$.

Observe that  if $\ell_i=-1$ then 
\begin{displaymath}
(g_i)^{\ell_i}=(\rho_{A'_i,A_i}\circ f_{A'_i, B'_i}\circ \rho _{A_i,A'_i})^{-1}
=\rho_{A'_i,A_i}\circ (f_{A'_i, B'_i})^{-1}\circ \rho _{A_i,A'_i}.
\end{displaymath}

So
\begin{multline*}
	(g_{i_0})^{\ell_{i_0}}\circ(g_{i_1})^{\ell_{i_1}}\circ
\dots\circ (g_{i_{k-1}})^{\ell_{i_{k-1}}}=\\
\rho_{A'_{i_0},A_{i_0}}\circ (f_{A'_{i_0}, B'_{i_0}})^{\ell_{i_0}}\circ 
\rho _{A_{i_0},A'_{i_0}}\circ 
\rho_{A'_{i_1},A_{i_1}}\circ (f_{A'_{i_1}, B'_{i_1}})^{\ell_{i_1}}\circ 
\rho _{A_{i_1},A'_{i_1}}\circ 
\end{multline*}

For $j<k$
let 
\begin{displaymath}
	\rho^*_j=\rho _{A_{i_j},A'_{i_j}}\circ 
	\rho_{A'_{i_{j+1}},A_{i_{j+1}}}.
\end{displaymath}
Observe that 
\begin{displaymath}
	\rho^*_j=
	\rho_{
	\rho_{A_{i_{j+1}},A'_{i_{j+1}}}[A_{i_{j}}\cap A_{i_{j+1}}],
	\rho _{A_{i_{j}},A'_{i_{j}}}[A_{i_{j}}\cap A_{i_{j+1}}] 
	}\in \mc S.
\end{displaymath}
(See Figure \ref{fg}.)

\begin{figure}[h]
\begin{tikzpicture}
	\coordinate (aj) at (5,0);
	\coordinate (ai) at (5,1);
\coordinate (an) at (5,0,5);

\coordinate (aip) at (0,2.5);
\coordinate (ajp) at (0,0);

\coordinate (ax)  at ($(aip)+(0.866,-0. 5)$);
\coordinate (ay)  at ($(ajp)+(0.866,0. 5)$);

\coordinate (al)  at ($(aip)+(0,-0. 5)$);
\coordinate (au)  at ($(ajp)+(0,0. 5)$);
\coordinate (a)  at ($(aj)+(0,0.5)$);

\coordinate (b)  at ($(a)+(2,0)$);

\path[draw,->, shorten >= 2mm, shorten <= 4mm]  (a) -- (al)  
node[above,pos=0.5, sloped]
{$\rho_{{A_{i_{j}}},{A'_{i_{j}}}} $};
\path[draw,->, shorten >= 4mm, shorten <= 1mm]  (au) -- (a)
node[below,pos=0.5, sloped]
{$\rho_{{A'_{i_{j+1}}},{A_{i_{j+1}}}} $};
\path[draw,->]  (au) -- (al) node[pos=0.5,left] {$\rho^*_j$};

\draw (ai)   circle (1cm) 
                  node[anchor=south] {$A_{i_j}$};
\draw (aj)   circle (1cm) 
				  node[anchor=north] {$A_{i_{j+1}}$};
\draw (ajp)   circle (1cm)[anchor=north] node {$A'_{i_{j+1}}$} ;
\draw (aip)   circle (1cm)[anchor=south] node {$A'_{i_{j}}$} ;

\draw[dotted, ->] (b) -- (a);
\node[right] at (b) {$A_{i_j}\cap A_{i_{j+1}}$};

\path[draw] (ax) arc (30:150:1); 
\path[draw] (ay) arc (-30:-150:1); 
\end{tikzpicture}
\caption{The function $\rho^*_j$}\label{fg}
\end{figure}

Thus 
\begin{multline*}
	(g_{i_0})^{\ell_{i_0}}\circ(g_{i_1})^{\ell_{i_1}}\circ
	\dots\circ (g_{i_{k-1}})^{\ell_{i_{k-1}}}=\\
	\rho_{A_{i_0},A'_{i_0}}\circ 
	(f_{A'_{i_0}, B'_{i_0}})^{\ell_0}\circ \rho^*_0\circ
	(f_{A'_{i_1}, B'_{i_1}})^{\ell_1}\circ \rho^*_1\circ \dots \\
	\circ 
	(f_{A'_{i_{k-1}},B'_{i_{k-1}}})^{\ell_{i_{k-1}}}\circ \rho_{A'_{i_{k-1}},A_{i_{k-1}}}. 
	\end{multline*}

Since $\rho_{A_\ell,A'_\ell}\restriction {\kappa}=\operatorname{id}\restriction 	{\kappa}$,
we have
\begin{multline*}
\bigl(	(g_{i_0})^{\ell_{i_0}}\circ(g_{i_1})^{\ell_{i_1}}\circ
	\dots\circ (g_{i_{k-1}})^{\ell_{i_{k-1}}}\bigr) 
	\cap{\kappa}\times {\kappa}\subs\\ 
	(f_{A'_{i_0}, B'_{i_0}})^{\ell_0}\circ \rho^*_0\circ
	(f_{A'_{i_1}, B'_{i_1}})^{\ell_1}\circ \rho^*_1\circ \dots \\
	\circ (f_{A'_{i_{k-1}},B'_{i_{k-1}}})^{\ell_{i_{k-1}}}
	\end{multline*}
But $(f_{A'_{i_0}, B'_{i_0}})^{\ell_0}\circ \rho^*_0\circ
(f_{A'_{i_1}, B'_{i_1}})^{\ell_1}\circ \rho^*_1\circ \dots 
\circ (f_{A'_{i_{k-1}},B'_{i_{k-1}}})^{\ell_{i_{k-1}}}=t[]$
for the  $\mc F\cup \mc S$-term 
$t=\<(f_{A'_{i_0}, B'_{i_0}})^{\ell_0},\rho^*_0,
(f_{A'_{i_1}, B'_{i_1}})^{\ell_1},\rho^*_1,\dots, 
(f_{A'_{i_{k-1}},B'_{i_{k-1}}})^{\ell_{i_{k-1}}}\>$.

Since there  are only finitely many sequences $i_0<\dots i_{k-1}<n$,
we obtain that 
$h\cap {\kappa}\times {\kappa}$ is covered by the union of finitely many 
$\mc F\cup \mc S$-terms.

But $y$ is not covered by the union of finitely many 
$\mc F\cup \mc S$-terms. So $y$ witnesses that $G$ is not 
${\kappa}$-transitive.
\end{proof}

\begin{proof}[Proof of the Key Lemma \ref{lm:keylemma}]
Write $\mc D'=\{\<A_{\alpha},B_{\alpha}\>:{\alpha}<{\kappa}^+\}$.

By transfinite induction, 
we define functions $\{f_{\alpha}:{\alpha}<{\kappa}^+\}$ such that 
taking 
\begin{displaymath}
\mc F_{<{\beta}}=\{f_{\gamma}:{\gamma}<{\beta}\}
\end{displaymath}
and 
\begin{displaymath}
	\mc S_{<{\beta}}=\{\trans{C_0}{C_1}: {\delta},{\gamma}< {\beta}, 
	C_0\in \itrace{\mc A}{A_{\delta}},
 C_1\in \itrace{\mc A}{A_{\gamma}},
 \rho_{C_0,C_1}[\mc A\restriction C_0]=\mc A\restriction C_1   \},	
	\end{displaymath}
we have	
\begin{enumerate}[(i)]
\item $f_{\alpha}\in \perm{A_{\alpha}}$,
\item $f_{\alpha}[B_{\alpha}]=K$,
\item 
if $\mc H$ is a finite collection of  
$\mc F_{<{\alpha}+1}\cup \mc S_{<\alpha+1}$-terms, 
then 
\begin{displaymath}
|y\setm \mc H[\ ]|={\kappa}.
\end{displaymath}
\end{enumerate}

Assume that we have  constructed $f_{\beta}$ for ${\beta}<{\alpha}$.
Then we have: 
\begin{equation}\tag{$*$}
\label{enum:star} \text{\it if $\mc H$ is a finite collection of  
$\mc F_{<\alpha}\cup \mc S_{<\alpha}$-terms, then } 
|y\setm \mc H[\ ]|={\kappa}.
\end{equation}

To continue the construction we need a bit more.
\begin{claim} 
	If $\mc H$ is a finite collection of  
	$\mc F_{<{\alpha}}\cup \mc S_{<\alpha+1}$-terms, then 
	\begin{displaymath}
	|y\setm \mc H[\ ]|={\kappa}.
	\end{displaymath}	
\end{claim}

\begin{proof}
First  observe that 
if ${\rho}_i={\rho}_{A_i,A^*_i}$ for $i<2$, then 
 \begin{equation}\tag{\ddag}\label{eq:dagg}
 	{\rho}_1\circ {\rho}_0={\rho}_{{\rho}_0^{-1}[A^*_0\cap A_1],{\rho}_1[A^*_0\cap A_1]}.
 \end{equation}

Let
\begin{displaymath}
t=\<t_0,t_1,\dots, t_n\>
\end{displaymath}
be an element of $\mc H$.
Since ${\rho}_{C_0,C_1}\restriction {\kappa} =\operatorname{id}\restriction {\kappa}$,  
$t[\ ]\cap 	{\kappa}\times 	{\kappa}=\<t_1,\dots t_n\>[\ ]\cap {\kappa}\times {\kappa}$ if $t_0\in \mc S_{<\alpha+1}$.
So we can assume that $t_0\in \mc F_{<{\alpha}}$.
Similar argument give that we can assume that $t_n\in \mc F_{<{\alpha}}$.

Now assume that
\begin{displaymath}
\<t_i,\dots , t_j\>=\<f_{\alpha_i},\rho_{C_{i+1},D_{i+1}},\rho_{C_{i+2},D_{i+2}},  \dots, \rho_{C_{j-1},D_{j-1}},f_{{\alpha}_j}\>
\end{displaymath} 
Then, by (\ref{eq:dagg})
\begin{displaymath}
	\rho_{C_{i+1},D_{i+1}}\circ\rho_{C_{i+2},D_{i+2}}\circ   \dots \circ \rho_{C_{j-1},D_{j-1}} = \rho_{E_i, E_j}.
\end{displaymath}
for some $E_i\in \trace {\mc A}{C_{i+1}}$
and $E_j\in \trace {\mc A}{D_{j-1}}$.

Thus we can assume that $j=i+2$ and  
\begin{displaymath}
\<t_i,t_{i+1}, t_{i+2}\>=\<f_{{\alpha}_0},{\rho}_{E_0,E_1},f_{{\alpha}_1}\>.
\end{displaymath}

Now 
\begin{displaymath}
	f_{{\alpha}_0}\circ {\rho}_{E_0,E_1}\circ f_{{\alpha}_1}=
	f_{{\alpha}_0}\circ {\rho}_{A_{{\alpha}_0}\cap E_0,A_{{\alpha}_1}\cap E_1}\circ f_{{\alpha}_1}
\end{displaymath}
and ${\rho}_{A_{{\alpha}_0}\cap E_0,A_{{\alpha}_1}\cap E_1}\in \mc S_{<{\alpha}}$.

Thus there is a  $\mc F_{<{\alpha}}\cup \mc S_{<\alpha}$-terms   $s_t$
such that 
\begin{displaymath}
t[\ ]\cap ({\kappa}\times {\kappa})=s_t[\ ]\cap ({\kappa}\times {\kappa}).
\end{displaymath}

Since $|y\setm \bigcup\{s_t[\ ]:t\in \mc H\}|={\kappa}$  by (\ref{enum:star}),
the Claim holds. 
\end{proof}

Since the claim holds, we can  apply Lemma \ref{lm:induction-step} for the family 
$\mc F =\mc F_{<{\alpha}}\cup \mc S_{{<\alpha}+1}$ 
to obtain $f_{\alpha}$ as  $g$. 

So we proved the Key Lemma \ref{lm:keylemma}.
\end{proof}

So we proved theorem \ref{tm:-main-uncountable-transitive-group}
\end{proof}

The following theorem is hidden in \cite{K}:
\begin{theorem}\label{tm:flat}
If ${\kappa}^{\omega}={\kappa}$, ${\lambda}={\kappa}^{+n}$ for some $n<{\omega}$,
and $\Box_{{\nu}}$ holds for each  ${\kappa}\le {\nu}<{\lambda}$,
then there is a cofinal, locally small family in $\br {\lambda};{\kappa};$.	
\end{theorem}

Indeed, in subsection 2.4 of \cite{K} the author defines the {\em weakly rounded}
subsets of ${\lambda}={\kappa}^{+n}$, in Lemma 2.4.1 he shows that the family of 
weakly rounded sets is cofinal, finally on page 52 he proves a Claim which 
clearly implies that the family of 
weakly rounded sets is locally small.   

Putting together Theorems  \ref{tm:-main-uncountable-transitive-group}
and  \ref{tm:flat} we obtain the following corollary.

\begin{corollary}\label{cor:largekappa}
	If ${\kappa}^{\omega}={\kappa}$, ${\lambda}={\kappa}^{+n}$ for some $n<{\omega}$,
	and $\Box_{{\nu}}$ holds for each  ${\kappa}\le {\nu}<{\lambda}$,
then there is a ${\kappa}$-homogeneous, but not ${\kappa}$-transitive 
permutation group on ${\lambda}$.		
\end{corollary}

\section{${\omega}$-homogeneous but not ${\omega}$-transitive permutation 
groups in the Cohen model}

For $f\in \perm {\kappa}$ let $\supp(f)=\{{\alpha}: f({\alpha})\ne 	{\alpha}\}$.
Write 
\begin{displaymath}
\parcperm {\lambda}{\omega}=\{f\in \perm {\lambda}: 
|\supp(f)|\le {\omega}\}.
\end{displaymath}

\begin{theorem}
\label{tm:ohnot_con}
If  
$P=\fin(2^{\omega},2)$ then
\begin{multline}\notag
V^P\models
\mbox{``for each 
${\lambda}\ge{{\omega}_1}$ there is an ${\omega}$-homogeneous} 
\\\text{ and ${\omega}$-intransitive permutation group on ${\lambda}$.''}
\end{multline}

\end{theorem}

The proof of this theorem is based on the following Lemma.

Let us recall that if  $g\in \perm {\omega_1}$ then 
$\extend{g}=g\cup (\operatorname{id}\restriction ({\lambda}\setminus{\omega_1}))$.

\begin{lemma}\label{lm:steppingup}  Assume that
$V_0\subs V_1$ are ZFC models and ${\lambda}\ge{\omega}_2$  is
a cardinal in $V_1$. If
\begin{enumerate}
\item $\forall X\in\big(\br {\lambda};{\omega};\big)^{V_1}\ $
$\exists Y\in\big(\br {\lambda};{\omega};\big)^{V_0}\ $
$X\subs Y$,
\item $V_1\models$
$G$ is an ${\omega}$-homogeneous permutation group on ${\omega}_1$,    
\\ 
\makebox[3cm]{} \ $G\supset{\parcperm {\omega_1}{\omega}}^{V_0}$, and  $r\in \perm{\omega}$ 
is $G$-large,
\end{enumerate}
then in $V_1$
the permutation group $G^*$ on ${\lambda}$ generated by 
\begin{displaymath}
\{\extend{g}:g\in G\}\cup\parcperm {\lambda}{\omega}^{V_0}
\end{displaymath}
is  ${\omega}$-homogeneous, and $r$ is   $G^*$-large. 
\end{lemma}

\begin{proof}
We will work in $V_1$.
 
First we show that $G^*$ is ${\omega}$-homogeneous.

If $X,Y\in\br{\lambda};{\omega};$ first pick 
$X_0,Y_0\in\br{\lambda};{\omega};\cap V_0$ with $X\subs X_0$ and
$Y\subs Y_0$ such that $|X_0\setminus X|=|Y_0\setminus Y|={{\omega}}$.
Fix $f,h\in\parcperm{\lambda}{\omega}^{ V_0}$ with $f[X_0]={{\omega}}$
and $h[Y_0]={\omega}$.
Since $G$ is ${\omega}$-homogeneous, there is  $g\in G$ with
$g\bigl [f[X]\bigr]=h[Y]$. Then $(h^{-1}\circ \extend g\circ f)[X]=Y$ and
$h^{-1}\circ \extend g\circ f\in G^*$.

Before proving  that $r$ is $G^*$-large we need some preparation.
Write
\begin{displaymath}
	G^+=\{g^+:g\in G\}.
\end{displaymath}

\begin{claim}\label{cl:hrestriction}
If $h_0, \dots h_k\in \parcperm{{\lambda}}{{\omega}}^{V_0}$
and $A\in \br {\omega}_1;{\omega};$ then there is 
$h\in \parcperm{{\omega}_1}{{\omega}}^{V_0}$
such that 
\begin{displaymath}
(h_0\circ\dots\circ h_k)\cap (A\times A)\subs h.
\end{displaymath}	
\end{claim}

\begin{proof}[Proof of the Claim \ref{cl:hrestriction}]
By (1) we can assume that $A\in V_0$, and so 
$h'=(h_0\circ\dots\circ h_k)\cap (A\times A)\in V_0$.
Since $h'$ is a countable injective function 
with $\dom(h')\cup \ran(h')\subs {\omega}_1$ it can be extended to 
a permutation $h\in  \parcperm{{\omega}_1}{{\omega}}^{V_0}$.
\end{proof}

If $\mc F$ is a set of functions,
let
\begin{displaymath}
\gen{\mc F}=\{f_0\circ\dots \circ f_{n-1}:n\in {\omega}, f_i\in \mc F 
\text{ for $i<n$}\}.
\end{displaymath}

	\begin{claim}\label{cl:starless}
		For each   $t\in \gen{G^+\cup \parcperm{{\lambda}}{{\omega}}^{V_0}}$ 		t
		there is a finite set 
		$\mc H\subs \gen{G\cup \parcperm{{\lambda}}{{\omega}}^{V_0}}$ 
		such that  
\begin{displaymath}
t\subs \bigcup \mc H.
\end{displaymath}
			\end{claim}

\begin{proof}[Proof of the Claim  \ref{cl:starless}]
If $t=f_0\circ\dots\circ f_{n-1}$, let
\begin{displaymath}
\mc H=\{\operatorname{id}_{\lambda}\}\cup \{f'_{i_0}\circ \dots \circ f'_{i_j}\circ \dots \circ f'_{i_k}: k\le n,
i_0<\dots <i_j<\dots <i_k<n\},
\end{displaymath}
where $f'_i=f_i$ if $f_i\in \parcperm{{\lambda}}{{\omega}}^{V_0},$
	and $f_i=g$ if $f_i=g^+$ for some $g\in G$,
and $\operatorname{id}_{\lambda}$ denotes the identity function on
	${\lambda}$.

Pick ${\alpha}\in {\lambda}$ such that $t({\alpha})\ne {\alpha}$.

Write
${\alpha}_n={\alpha}$ and ${\alpha}_i=f_i({\alpha}_{i+1})$
for $i=n-1,\dots, 0$. Let $0\le i_0<i_1\dots <i_\ell < n$
be the increasing enumeration of the set $\{i<n: {\alpha}_i\ne {\alpha}_{i+1}\}$.
Let $s=f'_{i_0}\circ\dots\circ f'_{i_\ell}$. Then
$s\in \mc H$ and  $s({\alpha})=t({\alpha})$.
\end{proof}

\begin{claim}\label{cl:lambdaoo}
	For each   $s\in \gen{G\cup \parcperm{{\lambda}}{{\omega}}^{V_0}}$
	and countable set $A\in \br {\omega}_1;{\omega};$
	there is  $u \in \gen{G\cup \parcperm{{\omega}_1}{{\omega}}^{V_0}}$
	such that 
	 $$s\cap (A\times A)\subs u.$$ 
		\end{claim}

\begin{proof}[Proof of the Claim \ref{cl:lambdaoo}]
Since both $G$ and $\parcperm{{\lambda}}{{\omega}}^{V_0}$
are groups we can assume that 
\begin{displaymath}
s=g_0\circ h_0\circ \dots\circ g_n\circ h_n,
\end{displaymath}
where $g_i\in G$ and $h_i\in \parcperm{{\lambda}}{{\omega}}^{V_0}$.

Write $A_{n}=A$, and let  $B_i=h_i[A_{i+1}]\cap {\omega}_1$
and $A_i=g_i[B_i]$ for $i=n-1,\dots,0$. 

By Claim \ref{cl:hrestriction} for each $i$ there is $h'_i\in \parcperm{{\omega}_1}{{\omega}}^{V_0}$
such that $h_i\cap (A_{i+1}\times B_i)\subs h_i'$.

Let $u=g_0\circ h_0'\circ \dots\circ  g_n\circ  h'_n$.

We show that $s\cap (A\times A)\subs u$.

Fix ${\alpha}\in A$.  Let ${\alpha}_n={\alpha}$ and for $i=n-1,\dots,0$
let ${\beta}_i=h_i({\alpha}_{i+1})$ and ${\alpha}_i=g_i({\beta}_i)$.
If $s({\alpha})$ is defined and $s({\alpha})\in A$, then 
for each $i<n$  we have ${\beta}_i\in B_i$ and $ {\alpha}_i\in A_i$,
and so $u({\alpha})$ is also defined and $u({\alpha})=s({\alpha}) $.
\end{proof}

Putting together Claims \ref{cl:starless} and \ref{cl:lambdaoo}
we obtain that 
\begin{claim}\label{cl:rlarge}
For each $g\in G^*$ there is a finite subset $H_g$ of $G$ such that 
$$g\cap({\omega}\times {\omega})\subs \bigcup\{h\restriction {\omega}: h\in H_g\}.$$  	
\end{claim}	

Claim \ref{cl:rlarge} yields that 
 $r$ is $G^*$-large.

So we proved the $G^*$ is ${\omega}$-intransitive which completes the 
proof of the lemma.
\end{proof}

By Lemma \ref{lm:steppingup}   the following theorem yields theorem
\ref{tm:ohnot_con}.
\begin{theorem} \label{tm:aux} If  $P=\fin(2^{\omega},2)$
then
$V^P\models$
{\em ``there is an ${\omega}$-homogeneous and  ${\omega}$-intransitive
permutation group $G$ on ${{\omega}_1}$ with
$G\supset\parcperm {\omega_1}{\omega}^V$''.}
 \end{theorem}

\begin{proof}
Given  sets $X$ and $Y$ let us denote by
$\bijp(X,Y)$  the set of all finite bijections between subsets of $X$
and $Y$.

We will define an iterated forcing system
with finite support
$$\<P_{\nu}:0\le{\nu}\le2^{\omega},\mc Q_{\nu}:-1\le{\nu}<2^{\omega}\>$$
and an increasing sequence of permutation groups
$\<G_{\nu}:\nu<{2^{\omega}}\>$,
$G_{\nu}\triangleleft\perm \omega^{V^{P_{\nu}}} $, simultaneously.

Take $G_0=\parcperm{{\omega}_1}{\omega}^V$ and 
$P_0=\mc Q_{-1}=\bijp({\omega},{\omega})$.
Denote by $r$ the generic  permutation of ${\omega}$  
given by the $V$-generic filter over $P_0$. 
By standard density arguments it is easy to see that
$r$ is $G_0$-large.
Now we  carry out the inductive construction as follows:
\begin{itemize}
\item for each ${\nu}<{{2^{\omega}}}$ we pick
$X_{\nu}$, $Y_{\nu},Z_{\nu}\in(\br{{\omega}_1};{\omega};)^{V^{P_{\nu}}}$ 
with $X_{\nu}\cup Y_{\nu}\subs Z_{\nu}$
and $|Z_{\nu}\setm X_{\nu}|=|Z\setm Y_{\nu}|={\omega}$,
\item put 
\begin{displaymath}
Q_{\nu}=\{p_0\cup p_1: p_o\in \bijp(X_{\nu},Y_{\nu}), p_1\in  
\bijp(Z_{\nu}\setm X_{\nu},Z_{\nu}\setm Y_{\nu})\},
\end{displaymath}
$\mc Q_{\nu}=\<Q_{\nu},\supset\>$
and $g_{\nu}=\bigcup{\mc G}_{\nu}$, where 
${\mc G}_{\nu}$ is the $\mc Q_{\nu}$-generic filter over $V^{P_{\nu}}$,
\item take   $G_{{\nu}+1}$ as the subgroup of
$\perm {{\omega}_1}^{V^{P_{{\nu}+1}}}$ generated by
$G_{\nu}\cup\{\extend {g_{\nu}}\}$.
\item for limit ${\nu}$ let 
$G_{\nu}=\bigcup_{{\zeta}<{\nu}}G_{\zeta}$.
\end{itemize}

We use a bookkeeping function to ensure
that every pair $X,Y\in(\br{{{\omega}}};{\omega};)^
{V^{P_{2^{\omega_1}}}}$   
will be chosen as $X_\nu$, $Y_\nu$ in some step.
Then $G=\bigcup_{{\nu}<2^{{\omega}}}G_{\nu}$ will
be ${\omega}$-homogeneous.

So the question is whether we guarantee that 
$r$ is $G_{\nu}$-large  during the induction.

If ${\nu}$ is a limit ordinal, then $G_{\nu}=\bigcup_{{\zeta}<{\nu}}G_{\zeta}$,
so if $r$ is $G_{\zeta}$-large for ${\zeta}<{\nu}$, then 
$r$ is $G_{\nu}$-large as well.

Assume now that $r$ is $G_{\nu}$-large and prove that 
$r$ is $G_{{\nu}+1}$-large as well.

The following lemma clearly implies this statement.
In this lemma we use some notations introduced
in Definition \ref{df:term-eval} in the previous section.
\begin{lemma}\label{lm:cohenextension}
If $\mc H$ is a finite set of $\fx{G_{\nu}}{x}$-terms, $p\in Q_{\nu}$,
$M$ is a natural number, 
then there is a condition $q\le p$ in $Q_{\nu}$
and there is ${\alpha}\in {\omega}\setm M$
such that $t[q]({\alpha})$ is defined for each $t\in \mc H$
and $t[q]({\alpha})\ne r({\alpha})$.
\end{lemma}

\begin{proof}[Proof of the lemma]
	We can assume that $\mc H$ is closed for subterms.

	We know that $|r\setm \bigcup \mc H[\ ]|={\omega}$ 
	because $r$ is $G_{\nu}$-large.

	Since $\mc H$ is closed for subterms, 
	$$y\cap \bigcup \mc H[\ ]=y\cap \bigcup \mc H[
		\operatorname{id}_{{\omega}_1\setm Z_{\nu}} ].$$
	Since $|p|<{\omega}$, we have 
	\begin{displaymath}
	|y\setm \bigcup \mc H[p\cup \operatorname{id}_{({\lambda}\setminus Z_{\nu})} ]|={\omega}.
	\end{displaymath}
	So we can pick ${\alpha}\in {\omega}\setm M$
	such that 
	 \begin{enumerate}[(i)]
		\item[$(*)$] for each $t\in \mc H$ either  
		$t[p\cup \operatorname{id}_{{\lambda}\setm Z_{\nu}}]({\alpha})$ 
		is undefined
	 or	$t[p\cup \operatorname{id}_{{\lambda}\setm Z_{\nu}}]({\alpha})
	 \ne r({\alpha})$.
		 \end{enumerate}
		
	Now in finitely many steps, using Lemma \ref{lm:induction-step-step},
	we can extend the function $p\in Q_{\nu}$ to a function 
	$q\in Q_{\nu}$ such that 
	\begin{enumerate}[(i)]
		\item[$(*)$] $t[{q}\cup \operatorname{id}_{{\lambda}\setm Z_{\nu}}]
		({\alpha})$ 
		is defined and 
	 $t[q\cup \operatorname{id}_{{\lambda}\setm Z_{\nu}}]({\alpha}_)\ne 
	 r({\alpha})$
		for each $t\in \mc H$.
		 \end{enumerate}
	
	Indeed, if $t[{q'}\cup \operatorname{id}_{{\lambda}\setm Z_{\nu}}]({\alpha})$ 
	is not defined, where $t=\<t_0,\dots, t_n\>$ then there is $i<n$
	such that  
	either 
	\begin{enumerate}[(1)]
	\item[] ${\zeta}_i=\<t_{i+1},\dots, t_n\>[{q'}\cup \operatorname{id}_{{\lambda}\setm Z_{\nu}}]({\alpha})$
	is defined, 
	$t_{i}=x$ and ${\zeta}_i\notin \dom(q')$ 
	\end{enumerate}
	or  
	\begin{enumerate}[(1)]
		\item[] ${\zeta}=\<t_{i+1},\dots, t_n\>[{g'}\cup \operatorname{id}_{{\lambda}\setm Z_{\nu}}]({\alpha})$
		is defined, $t_{i}=x^{-1}$ and ${\zeta}_i\notin \ran(q').$ 
		\end{enumerate}
	In both cases, using Lemma \ref{lm:induction-step-step}, 
	we can extend $q'$ to $q''$ such that
	$\<t_i,\dots, t_n\>[{q''}\cup \operatorname{id}_{{\lambda}\setm Z_{\nu}}]({\alpha})$
	is defined and $\<{\alpha},r({\alpha})\>\notin \mc H[q''\cup id_{{\lambda}\setm Z_{\nu}}]$.	
	So we proved Lemma \ref{lm:cohenextension}.
\end{proof}

So  $r$ is  $G_{{\nu}+1}$-large.

Thus, by transfinite induction,  we proved that $r$ is $G$-large
which completes the proof of the theorem.
\end{proof}

\end{document}